\documentclass{l4dc2025}

\title[Adaptive Koopman MPC]{Model Predictive Control of Nonlinear Dynamics\\ Using Online Adaptive Koopman Operators}
\usepackage{times}
\usepackage{graphicx}
\usepackage{amsmath, amssymb, cases, comment, color}

\newcommand{\tr}{\mathsf{T}}
\newtheorem{prob}{Problem}
\newtheorem{prop}{Proposition}
\newtheorem{thm}{Theorem}

\author{%
 \Name{Daisuke Uchida} \Email{duchida@umich.edu}\vspace*{-1mm}
 \AND
 \Name{Karthik Duraisamy} \Email{kdur@umich.edu}\\
 \addr Department of Aerospace Engineering, University of Michigan, MI 48109, USA
}

\usepackage{setspace}
\begin{document}

\maketitle

\begin{abstract}%
This paper develops a methodology for adaptive data-driven Model Predictive Control (MPC) using Koopman operators.
While MPC is ubiquitous in various fields of engineering, the controller performance can deteriorate if the modeling error between the control model and the true dynamics persists, which may often be the case with complex nonlinear dynamics. 
Adaptive MPC techniques learn models online such that the controller can compensate for the modeling error by incorporating newly available data. 
We utilize the Koopman operator framework to formulate an adaptive MPC technique  that corrects for model discrepancies in  a computationally efficient manner by virtue of  convex optimization.
With the use of neural networks to learn embedding spaces, Koopman operator models  enable accurate dynamics modeling. Such complex model forms, however, often lead to unstable online learning.
To this end, the proposed method utilizes the soft update of target networks, a technique used in stabilization of model learning in Reinforcement Learning (RL). 
Also, we provide a discussion on which parameters to be chosen as online updated parameters based on a specific description of linear embedding models.
Numerical simulations on a canonical nonlinear dynamical system show that the proposed method performs favorably compared to other data-driven MPC methods while achieving superior computational efficiency through the utilization of Koopman operators.

\end{abstract}

\begin{keywords}%
  Koopman Operator, Data-driven Control, Model Predictive Control
\end{keywords}

\section{Introduction}
\vspace{-2mm}
Model Predictive Control (MPC) has emerged as a powerful framework for solving complex control challenges across diverse engineering applications, from process industries to robotics and autonomous systems. At the heart of MPC lies its ability to optimize control decisions based on predicted future system behavior. This predictive capability critically depends on the accuracy of the underlying control model. Yet for complex systems, the mathematical models used for control often deviate from the true system dynamics. To address this limitation, researchers have developed adaptive and data-driven MPC approaches that explicitly account for modeling uncertainties, leading to more robust control performance
(\cite{adaptive_MPC_constrained_system, GP_based_periodic_error_correction_MPC, robust_constrained_learning_based_NMPC, learning_based_mpc_vision_based_robot}). 
For instance, the online adaptation procedure can be formulated such that the residual dynamics, which is the difference between the control model and the true dynamics, is learned online by fitting non-parametric or parametric function approximators such as Gaussian processes (\cite{GP_MPC}) and random Fourier features (\cite{RFF_MPC}).

Koopman operator theory has gained popularity in recent years for its utility as an alternative approach to describing nonlinear dynamics (\cite{Data_driven_book,Koopman_book,pan2024lifting}). 
Specifically, nonlinear dynamics can be represented as linear ones in the embedded space of feature maps, and several computational methods have been developed to obtain finite-dimensional approximations of Koopman operators, which are then utilized for prediction and control. 
The Koopman operator framework offers a potentially powerful advantage: it enables the transformation of nonlinear dynamical systems into linear representations through data-driven methods. This linearization affords access to the extensive theoretical machinery and computational tools developed for linear control systems, while preserving the ability to handle underlying nonlinear dynamics (\cite{KORDA_Koopman_MPC, MPC_for_PDE}). 
While Koopman operator-based models can be learned by linear regression types of methods such as Extended Dynamic Mode Decomposition (EDMD) (\cite{Williams2015}), utilization of neural networks to learn (in contrast to prescribing from a dictionary) feature spaces has been shown to be promising for complex nonlinear dynamics and incorporated into data-driven Koopman operator-based control (\cite{DeSKO,Deep_Koopman_vehicles,control_aware_Koopman,Physics-based_robabilistic_learning}).

While the use of Koopman operators enables expressive and flexible modeling for data-driven control,  modeling errors may arise due to several possible factors, e.g., lack of data quantity/quality, inadequate model structures, etc.
Whereas there are several methods to tackle this issue from the control theoretic viewpoints (\cite{tube_based_MPC, handling_plant_model_mismatch, data-driven_Koopman_H2}), most of them are based on EDMD-type models. 
Also, while \cite{DeSKO} develops a model uncertainty-aware Koopman MPC with the use of probabilistic neural networks and \cite{control_aware_Koopman} proposes a model refinement technique to handle the modeling error of neural network-based Koopman models in the context of control, there is a relative scarcity of exploration of online update methods of such control models.

In this paper, we propose an online adaptation method for Koopman operator-based MPC to avoid performance deterioration due to the modeling error.
Whereas \cite{online_DMD, online_DMD_fluid, online_DMD_power_system, online_DMD_EEG} explore the idea of online adaptation in the context of Koopman operator-based computational modeling, they do not assume controller design problems. 
\cite{online_DMD_seperation_control} presents an adaptive control of flow separation based on the online dynamic mode decomposition, which is an EDMD-type model without nonlinear feature maps. 
\cite{adaptive_Koopman_matrices} also develops an adaptive Koopman MPC, in which the linear operator $[A\ B]$ in \eqref{eq. linear embedding model; model dynamics in the embedded space} is updated to $[A+\Delta A\ B+\Delta B]$ s.t. $\Delta A$ and $\Delta B$ are parameterized by additional neural networks, which are trained online w.r.t. an adaptation loss function. 
On the other hand, the proposed method provides a more tractable yet  effective online adaptation procedure. 
We only use a single loss function to train neural networks throughout offline and online model learning, which results in fewer hyperparameters and less complexity of the model learning.
At the same time, the proposed method enables flexible online model learning since it allows an adaptation of the feature maps in addition to the linear operator $[A\ B]$.
Considering that model learning involving deep neural networks results in high-dimensional, non-convex optimizations and typically becomes unstable, we adopt the soft update of target networks (\cite{DDPG}), a common technique to stabilize learning used in Reinforcement Learning (RL). 
Also, we provide a discussion on which parameters should be prioritized in the online update procedure based on a specific formulation of linear embedding models in \cite{Koopman_form} to further improve the stability and robustness of online learning. The overview of the proposed method is shown in Fig. \ref{fig. concept}.

In Section \ref{sec. problem setup}, MPC with Koopman operator-based linear embedding models is presented. This is  followed by the formulation of offline model learning based on the Koopman operator framework in Section \ref{sec. offline learning}. The proposed method is formalized in Section \ref{sec. proposed adaptive Koopman MPC} and numerical evaluations are provided in Section \ref{sec. numerical example}.

\begin{figure}
    \centering
    \includegraphics[width=0.9\linewidth]{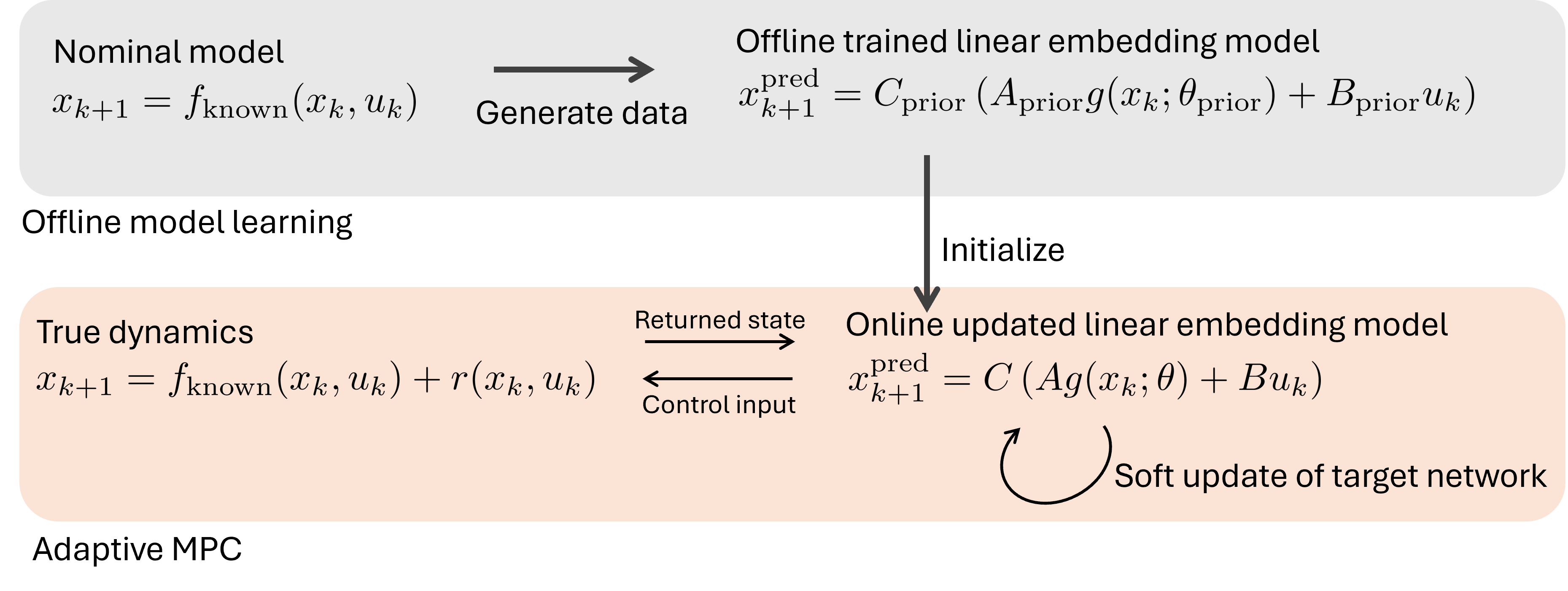}
    \caption{Proposed adaptive Koopman MPC with soft update. See Algorithm \ref{algorithm proposed method} for details.}
    \label{fig. concept}
\end{figure}

\vspace{-3mm}
\section{Problem Setup}
\vspace{-2mm}
\label{sec. problem setup}
\subsection{Model Predictive Control}
\vspace{-2mm}
We consider the problem of designing controllers for a discrete-time dynamics:
\vspace{-2mm}
\begin{align}
    x_{k+1}=f(x_k,u_k),
\end{align}
where $x_k\in \mathcal{X}\subseteq \mathbb{R}^n$ and $u_k\in \mathcal{U}\subseteq \mathbb{R}^p$ denote the state and the control input, respectively, and $f:\mathbb{R}^n\times \mathbb{R}^p\rightarrow \mathbb{R}^n$ is a possibly nonlinear mapping.
The control objective is supposed to minimize a quadratic cost $J$ at each time step, which is defined by
\vspace{-3mm}
\begin{align}
    J:=\sum_{k=0}^{H+1} 
    \left\{
        (x_k - x^\text{ref}_k)^\tr Q_\text{state} (x_k - x^\text{ref}_k) + u_k^\tr R u_k
    \right\},
    \label{eq. cost function w.r.t. the original state}
\end{align}
where $x_0$ denotes the state at the current time step, and $H$, $x^\text{ref}_k$, $Q_\text{state}$, and $R$ are a look-ahead horizon, a reference signal, and weight matrices w.r.t. the state and the control input, respectively.

It is assumed that while $f$ is unknown, we are given prior information about the dynamics in the form of nominal model $x_{k+1}=f_\text{known}(x_k,u_k)$ so that the true system is decomposed into:
\begin{align}
    x_{k+1} = f_\text{known}(x_k,u_k) + r(x_k,u_k),
    \label{eq. true dynamics decomposition}
\end{align}
where $r(x_k,u_k):=f(x_k,u_k) - f_\text{known}(x_k,u_k)$ is the residual dynamics.
This is a typical scenario in engineering applications such as robotics, and controller design under the unknown component $r$ is a problem that is actively being explored (\cite{safe_non_stochastic_control, neural_lander, safe_control_LTV}).

In MPC, the controller determines the optimal control input by minimizing a predefined cost function subject to the system's dynamic constraints as follows.
\begin{prob}
    \label{prob. general form of MPC}
    (Model Predictive Control with Quadratic Cost) \\
    \rm{}
    Given a current state $\xi_0$ of control model, apply the first element $u_0$ of the solution to the problem: 
    \begin{align}
    &\underset{u_0, u_1,\cdots, u_H}{\text{min}}
    \sum_{k=0}^{H+1} 
    \left\{
        (\xi_k - \xi^\text{ref}_k)^\tr Q (\xi_k - \xi^\text{ref}_k) + u_k^\tr R u_k
    \right\},
    &
    \label{eq. general MPC; quadratic cost}
\\
    &
    \text{subject to:}
    \ \ 
    \xi_{k+1}=\mathcal{F}(\xi_k,u_k,k),
    &
    \label{eq. general MPC; control model}
    \end{align}
    where a possibly time-varying mapping $\mathcal{F}:\mathbb{R}^N\times \mathbb{R}^p\times \mathbb{Z}_{\geq0}\rightarrow \mathbb{R}^N$ denotes a control model with the state $\xi_k$ and the control input $u_k$, and $\xi^\text{ref}_k$ is a reference signal. Weight matrices $Q$ and $R$ are assumed to be positive definite.
\end{prob}

\vspace{-2mm}
A simplest choice of the control model $\mathcal{F}$ is the nominal model so that $\xi_k\in \mathbb{R}^n$ ($N:=n$), $\mathcal{F}(\xi_k,u_k,k):=f_\text{known}(\xi_k,u_k)$, $\xi^\text{ref}_k:=x^\text{ref}_k$, and $Q:=Q_\text{state}$, in which case we call Problem \ref{prob. general form of MPC} \textit{nominal MPC}.
Since the cost function \eqref{eq. general MPC; quadratic cost} is quadratic, Problem \ref{prob. general form of MPC} becomes a convex optimization if the control model $\mathcal{F}$ is linear.
Note that the controller performance may be degraded due to the discrepancy between the control model $\mathcal{F}$ and the true dynamics $f$.

\vspace{-2mm}
\subsection{Linear Embedding Model}
\vspace{-2mm}
\label{sec. linear embedding model}
In the proposed method, we use a linear embedding model to derive the control model $\mathcal{F}$ for Problem \ref{prob. general form of MPC}. Given a state $x_k\in\mathcal{X}$ and a control input $u_k\in\mathcal{U}$ of the true dynamics \eqref{eq. true dynamics decomposition}, a linear embedding model outputs a predicted state $x^\text{pred}_{k+1}$ at the next time step s.t.
\vspace{-2mm}
\begin{numcases}{
    x^\text{pred}_{k+1} = C
    \left\{
        A g(x_k) + B u_k
    \right\}
    \Leftrightarrow 
}
    g^+ = A g(x_k) + B u_k,
    \label{eq. linear embedding model; model dynamics in the embedded space}
\\
    x^\text{pred}_{k+1} = C g^+,
    \label{eq. linear embedding model; decoder}
\end{numcases}
where $g:\mathcal{X}\rightarrow \mathbb{R}^N$ is a vector-valued function called feature maps and $A\in \mathbb{R}^{N\times N}$, $B\in \mathbb{R}^{N\times p}$, and $C\in \mathbb{R}^{n\times N}$ are matrices. 
The original state $x_k\in\mathbb{R}^n$ is first embedded onto a space $\mathbb{R}^N$ through the feature maps $g$, and the model state $g(x_k)$ is advanced by one step by the linear operator $[A\,B]$ to yield the prediction $g^+$ in \eqref{eq. linear embedding model; model dynamics in the embedded space}. 
The decoder $C$ then projects $g^+$ back onto the original state space $\mathbb{R}^n$ to make the state prediction $x^\text{pred}_{k+1}$ in \eqref{eq. linear embedding model; decoder}.
A sufficient condition for a linear embedding model to reconstruct the true dynamics is given by the following.

\begin{prop}
    \label{prop. conditions to make Koopman MPC aligned with original cost}
    \rm{}
    Consider a linear embedding model \eqref{eq. linear embedding model; model dynamics in the embedded space}, \eqref{eq. linear embedding model; decoder}.
    For a state-input triplet $(x_k,u_k,x_{k+1})$ of the true dynamics s.t. $x_k\in\mathcal{X},u_k\in\mathcal{U},x_{k+1}=f(x_k,u_k)$, a relation $x^\text{pred}_{k+1}=x_{k+1}$ holds if  
    \begin{align}
        &g(x_{k+1}) = A g(x_k) + B u_k,&
        \label{eq. condition on linear embedding model; invariance proximity}
    \\
        &x_{k+1} = Cg(x_{k+1}).&
        \label{eq. condition on linear embedding model; decoding error}
    \end{align}
\end{prop}
\vspace{-2mm}
\begin{proof}
    This follows directly from the definitions.
\end{proof}

\vspace{-2mm}
A major advantage of using a linear embedding model is its utility as a linear control model in the embedded space.
In Problem \ref{prob. general form of MPC}, consider a control model:
\vspace{-2mm}
\begin{align}
    \xi_{k+1} = A \xi_k + B u_k, \ \xi_0:=g(x_0)\in \mathbb{R}^N,
    \label{eq. Koopman control model}
\end{align}
with the reference signal and the weight matrix defined as $\xi^\text{ref}_k:=g(x^\text{ref}_k)$ and $Q:=C^\tr Q_\text{state}C$, respectively.
While the system \eqref{eq. Koopman control model} is no longer defined in the original state space $\mathcal{X}\subseteq \mathbb{R}^n$, an MPC solution in this setting still leads to an optimal control input w.r.t. $J$ in \eqref{eq. cost function w.r.t. the original state} on some conditions.
Specifically, if \eqref{eq. condition on linear embedding model; invariance proximity} and \eqref{eq. condition on linear embedding model; decoding error} hold for $\forall x_k\in\mathcal{X},\forall u_k\in\mathcal{U}$, a solution to Problem \ref{prob. general form of MPC} also minimizes the cost $J$ since 
\vspace{-2mm}
\begin{align}
    (\xi_k - \xi^\text{ref}_k)^\tr Q (\xi_k - \xi^\text{ref}_k)
    =
    (x_k - x^\text{ref}_k)^\tr Q_\text{state} (x_k - x^\text{ref}_k),
\end{align}
where $\xi_k=g(x_k)$ and $x_k=Cg(x_k)$ are substituted.

This class of MPC is often referred to as Koopman MPC in the literature since the linear operator $[A\, B]$ in \eqref{eq. linear embedding model; model dynamics in the embedded space} can be considered a finite-dimensional approximation of a Koopman operator (\cite{KORDA_Koopman_MPC}), as described in the next section.
MPC methods based on the Koopman operator formalism have  shown promise via their computational efficiency and control performance in various applications.
Specifically, Koopman MPC will result in faster execution times than most nonlinear MPC methods since the optimization becomes convex, and the validity of the linear control model \eqref{eq. Koopman control model} may be even established in case the true dynamics \eqref{eq. true dynamics decomposition} is nonlinear if the model parameters are learned appropriately (see the next section for details).
Also,  a linear embedding model can be computed in a fully data-driven manner, i.e., its model parameters can be determined by only using time-series data sampled from either the true dynamics or a simulator of a nominal model. 
In this paper, we refer to this type of controller design \textit{Koopman MPC}.
\vspace{-2mm}
\begin{prob}
    \label{prob. Koopman MPC}
    (Koopman MPC)\\
    \rm{}
    Given a current state $x_0$ and a linear embedding model \eqref{eq. linear embedding model; model dynamics in the embedded space}, \eqref{eq. linear embedding model; decoder}, solve Problem \ref{prob. general form of MPC} with the initial condition, the control model, and the weight matrix chosen as:
    \begin{align}
        &
        \xi_0:=g(x_0),
        &
    \end{align}
    \begin{align}
        &\xi_{k+1}=
        A \xi_k + B u_k,\text{ and}&
    \\
        &
        Q:=C^\tr Q_\text{state} C.
        &
    \end{align}
\end{prob}

In the proposed method, we employ the Koopman MPC as the baseline method of the control strategy with the use of a nominal model $f_\text{known}$, and compensate for the modeling error using an online  update while retaining the advantages of linear embedding models.
\vspace{-2mm}
\section{Offline Learning of Linear Embedding Models}
\label{sec. offline learning}
\vspace{-2mm}
The model parameters of the linear embedding model \eqref{eq. linear embedding model; model dynamics in the embedded space},\eqref{eq. linear embedding model; decoder} are the feature maps $g:\mathcal{X}\rightarrow \mathbb{R}^N$ and matrices $A,B,$ and $C$. 
In this section, we formulate a problem of learning these parameters in an offline manner, where it is assumed that only data samples generated by a nominal model $f_\text{known}$ are available but we do not have access to data from the true system \eqref{eq. true dynamics decomposition}.
\vspace{-2mm}
\subsection{Koopman Operators}
\vspace{-2mm}
Koopman operators characterize the time evolution of dynamical systems by acting as composition operators on function spaces, enabling the analysis of discrete-time dynamics through functional transformations. These operators share a fundamental mathematical connection with the linear embedding matrices $[A,B]$, providing a theoretical foundation for transforming nonlinear dynamics into linear representations.
For instance, given  autonomous dynamics $x_{k+1}=f_a(x_k)$, $x_k\in\mathbb{R}^n$, and a function $g:\mathbb{R}^n\rightarrow \mathbb{R}$ s.t. $g\in \mathcal{G}$ where $\mathcal{G}$ is some function space, the Koopman operator $\mathcal{K}$ associated with this system is defined as $\mathcal{K}:\mathcal{G}\rightarrow \mathcal{G}:g\mapsto g\circ f_a$ on the assumption that $g\circ f_a\in \mathcal{G}$, $\forall g\in \mathcal{G}$.
This corresponds to time evolution of the dynamics $x_{k+1}=f_a(x_k)$ through the function $g$ since 
\vspace{-2mm}
\begin{align}
    g(x_{k+1}) = g(f_a(x_k)) = (g\circ f_a)(x_k) 
    \overset{\text{def.}}{=}
    (\mathcal{K}g)(x_k).
    \label{eq. description of autonomous dynamics with a Koopman operator}
\end{align}

Note that $g$ corresponds to a feature map in our formulation in Section \ref{sec. linear embedding model}.
As $\mathcal{K}$ is a composition operator, it is easily confirmed that Koopman operators are linear operators. 
A major difference between the two descriptions of the dynamics is that the time evolution of \eqref{eq. description of autonomous dynamics with a Koopman operator} is governed linearly by $\mathcal{K}$ whereas $f_a$ is possibly nonlinear, which naturally leads to an idea of deriving linear models using $\mathcal{K}$.
Noticing that Koopman operators are, however, infinite-dimensional in general since they are defined on function spaces, this is realized by a finite-dimensional version or its approximation of $\mathcal{K}$. 
An exact finite-dimensional version of the Koopman operator $\mathcal{K}$ exists if and only if we can find functions $g$ that span an invariant subspace, as shown in the following proposition. 
\begin{prop}
    \label{prop. invariance property of Koopman operators; autonomous}
    \rm{}
    Given $N$ functions $g_i\in \mathcal{G}$, $i=1,\cdots,N$, there exists a matrix $K\in\mathbb{R}^{N\times N}$ s.t. 
    \begin{align}
        \label{eq. finite dimensional version of Koopman operator}
        [\mathcal{K}g_1\cdots \mathcal{K}g_N]^\tr = K [g_1\cdots g_N]^\tr, 
    \end{align}
    if and only if $\text{span}(g_1,\cdots,g_N)$ is an invariant subspace under the action of $\mathcal{K}$, i.e., \\$\mathcal{K}g\in \text{span}(g_1,\cdots,g_N)$ for $\forall g\in \text{span}(g_1,\cdots,g_N)$.
\end{prop}
\begin{proof}
    For instance, see \cite{control_aware_Koopman}.
\end{proof}

In practice, finding an invariant subspace is not trivial, and a finite-dimensional approximation of $\mathcal{K}$ is computed from data samples. 
Extended Dynamic Mode Decomposition (EDMD) (\cite{Williams2015}) is one  such method using linear regression and user-specified feature maps $g_i$.
While EDMD enables simple and tractable model training, it may not be adequately  expressive  for complex nonlinear systems.
On the other hand, joint learning of $g_i$ and $K$ can achieve more accurate models since $g_i$ is also learned on the training data along with the matrix $K$, and parameterizing $g_i$ by neural networks has been shown to be successful in a wide range of problems (\cite{Deep_learning_universal_linear_embeddings, Learning_Koopman_Invariant_Subspaces, Physics-based_robabilistic_learning}).

For general non-autonomous dynamics \eqref{eq. true dynamics decomposition} with the control input, corresponding Koopman operators are defined in a similar manner but with a sequence of control inputs included (\cite{KORDA_Koopman_MPC}). 
For the space of input sequences:
$
	l(\mathcal{U}):=
	\{
		(u_0,u_{1},\cdots)\mid u_k\in \mathcal{U}, \forall k
	\}
$, 
consider a mapping $\hat{f}:\mathcal{X}\times l(\mathcal{U})\rightarrow \mathcal{X}\times l(\mathcal{U}):(x,(u_0,u_1,\cdots))\mapsto (f(x,u_0), (u_1,u_2,\cdots))$.
Also, let $\hat{g}:\mathcal{X}\times l(\mathcal{U})\rightarrow \mathbb{R}$ be a function from an extended space $\mathcal{X}\times l(\mathcal{U})$ to $\mathbb{R}$.
Then, the Koopman operator associated with \eqref{eq. true dynamics decomposition} is defined as a linear operator $\mathcal{K}:\hat{\mathcal{G}}\rightarrow \hat{\mathcal{G}}:\hat{g}\mapsto \hat{g}\circ \hat{f}$ s.t. $\hat{\mathcal{G}}$ is a function space to which $\hat{g}$ belongs and the dynamics along with a sequence $(u_k,u_{k+1},\cdots)$ of control inputs is represented by
\vspace{-2mm}
\begin{align}
	\hat{g}(x_{k+1},(u_{k+1},u_{k+2},\cdots))&=
	(\hat{g}\circ \hat{f})(x_k,(u_k,u_{k+1},\cdots))=
	(\mathcal{K}\hat{g})(x_k,(u_k,u_{k+1},\cdots)).
	&\
\end{align}

It can be easily verified that Proposition \ref{prop. invariance property of Koopman operators; autonomous} also holds for the non-autonomous case.
If we consider the following $N+p$ functions $\hat{g}_i$ of specific forms:
\vspace{-2mm}
\begin{align}
	&[\hat{g}_1(x_k,(u_k,u_{k+1},\cdots))\cdots \hat{g}_{N+p}(x_k,(u_k,u_{k+1},\cdots))]^\tr
	=[g_1(x_k)\cdots g_{N}(x_k)\ u_k^\tr]^\tr, 
	&
\end{align}
the first $N$ rows of \eqref{eq. finite dimensional version of Koopman operator} reads
$g(x_{k+1}) = A g(x_k) + Bu_k$, where $[A\ B]\in \mathbb{R}^{{N}\times (N+p)}$ denotes the first $N$ rows of $K$.
Therefore, the condition \eqref{eq. condition on linear embedding model; invariance proximity} is ensured by choosing $[A\,B]$ as a finite-dimensional Koopman operator acting on an invariant subspace.
As is the case with the autonomous setting, finding such a subspace is not trivial in practice and a finite-dimensional approximation may be computed as $[A\,B]$ by either EDMD-type methods or joint learning of $g_i$ and $[A\,B]$.
\vspace{-2mm}
\subsection{Offline Learning Procedure}
\vspace{-2mm}
In the proposed method, we adopt the joint learning of feature maps $g_i$ and matrices $A$, $B$, and $C$ with the use of neural networks, which is formulated as follows.
\begin{prob}
        (Offline Learning Using Nominal Model)\\
	\rm{}
	Let $g(\cdot;\theta):\mathcal{X}\rightarrow \mathbb{R}^{N}$ be a neural network characterized by parameters $\theta$. 
	Find $\theta$, $A\in \mathbb{R}^{N\times N}$, $B\in \mathbb{R}^{N\times p}$, and $C\in \mathbb{R}^{n\times N}$ that minimize the loss function:
    \vspace{-2mm}
	\begin{align}
		&\sum_{i} 
        \left(
		\lambda_1
		\left\|
			Ag(x_i;\theta) + Bu_i - g(y_i;\theta)
		\right\|_2^2
		+
		\lambda_2
		\left\|
		C(Ag(x_i;\theta) + Bu_i) - y_i
		\right\|_2^2
        \right)
        ,&
		\label{eq. loss function}
	\end{align}
	where the data set is given in the form $\mathcal{D}:=\{ (x_i,u_i,y_i)\mid y_i=f_\text{known}(x_i,u_i) \}$ and $\lambda_1,\lambda_2\in \mathbb{R}$ are hyperparameters.
	\label{prob. Offline Model Learning}
\end{prob}

The first and second terms of the loss function \eqref{eq. loss function} are responsible for (approximately) ensuring the conditions \eqref{eq. condition on linear embedding model; invariance proximity} and \eqref{eq. condition on linear embedding model; decoding error}, respectively. 

\vspace{-3mm}
\subsection{Data Generation Using MPC Simulations}
\vspace{-2mm}
In Problem \ref{prob. Offline Model Learning}, how the control inputs $u_i$ in the dataset $\mathcal{D}$ are generated determines the quality of data and therefore has a significant influence on the learning results.
A typical strategy is sampling both states and inputs from some distributions assuming that $x_i,u_i$ are i.i.d. random variables, in which case the loss function \eqref{eq. loss function} will converge to a more general characteristic such as an $L_2$ norm of the modeling error as the number of data samples tends to infinity (e.g., \cite{control_aware_Koopman}).
However, it is challenging to sample  the product space $\mathcal{X}\times \mathcal{U}$ adequately unless the dimensions of the state and the control input are sufficiently small, and this sampling strategy may not necessarily result in accurate and unbiased model predictions in practice.

As an alternative approach, we utilize MPC simulations of the nominal model $f_\text{known}$ to generate data samples. 
Specifically, the dataset $\mathcal{D}$ in Problem \ref{prob. Offline Model Learning} consists of a collection of trajectories $(x_0,u_0,x_1,u_1,\cdots)$ s.t. $x_0$ is randomly sampled, and $u_k$ and $x_{k+1}$ are a solution of the nominal MPC (Problem \ref{prob. general form of MPC} with $\mathcal{F}(\xi_k,u_k,k):=f_\text{known}(\xi_k,u_k)$) and its corresponding next state $x_{k+1}=f_\text{known}(x_k,u_k)$, respectively. 
The main intent is to selectively learn relevant regimes of dynamics by the use of MPC simulations of the nominal model so that the controller performance will be improved.
A similar approach is also employed in \cite{Koopman_legged_robots}.
\vspace{-3mm}
\section{Adaptive Koopman MPC}
\vspace{-2mm}
\label{sec. proposed adaptive Koopman MPC}
While the Koopman MPC with a linear embedding model learned by Problem \ref{prob. Offline Model Learning} is expected to perform well if the nominal model $f_\text{known}$ is sufficiently close to the true dynamics, 
updating the control model online will further improve the performance or foster robustness in case there exists a large discrepancy between the nominal model and the true dynamics.

Let $\Theta:=\{ A,B,C,\theta \}$ be the model parameters of a linear embedding model \eqref{eq. linear embedding model; model dynamics in the embedded space}, \eqref{eq. linear embedding model; decoder}.
Also, we use a notation $\text{LinearMPCSolver}(g(x_k);\Theta)$ to denote a solution of the Koopman MPC (Problem \ref{prob. Koopman MPC}) given a current state $x_k$.
Assuming that a new data sample $x_{k+1}$ is available at time $k+1$ s.t. $x_{k+1}=f(x_k,u_k)$ where $u_k=\text{LinearMPCSolver}(g(x_k);\Theta)$, we add $x_{k+1}$ to a relay buffer, from which a data batch is sampled at each time step to update $\Theta$ in an online manner. 
Similar to Problem \ref{prob. Offline Model Learning}, the model parameters are updated online by minimizing \eqref{eq. loss function}.

It is, however, known that using neural networks as function approximators often makes learning processes unstable due to its non-convexity and high-dimensionality.
To this end, the proposed method adopts the soft update of target networks (\cite{DDPG}).
A target network is paired with a main network and its parameters are slowly updated towards that of the main network, by which abrupt changes in the outputs will be avoided to enhance the stability of the learning process.

In the proposed method, we initialize both the main and the target networks (two linear embedding models with the model parameters labeled $\Theta$ and $\Theta_\text{target}$, respectively) by the offline model from Problem \ref{prob. Offline Model Learning}. 
The target network is then updated at each time step by an interpolation $\Theta_\text{target}\leftarrow \tau\Theta+(1-\tau)\Theta_\text{target}$, where $\tau$ is a hyperparameter to adjust the smoothness of the update.
The actual control input is computed by the Koopman MPC solver with the slowly changing model parameters $\Theta_\text{target}$ so that we can suppress undesirable behavior of the closed-loop dynamics due to large fluctuations of the control model. 
The proposed \textit{Adaptive Koopman MPC with Soft Update} is summarized in Algorithm \ref{algorithm proposed method}.

\begin{algorithm2e}[t]
        \label{algorithm proposed method}
	\caption{Adaptive Koopman MPC with Soft Update}
		\textbf{Require:} Prior Model or simulator $f_\text{known}$ of known dynamics\;
            \vspace{2mm}
		 \textbf{Step 1: Train an offline model}\; 
             Simulate $f_\text{known}$ and collect data:\\ $\mathcal{D}_\text{known}=\{ (x,u,y)\mid y=f_\text{known}(x,u), u=\text{MPCSolver}(x;f_\text{known}) \}$\; 
            \vspace{2mm}
            Solve Problem \ref{prob. Offline Model Learning} to train a linear embedding model:
            \\ $x^\text{pred}_{k+1}=C_\text{prior} (A_\text{prior}g_\text{prior}(x_{k};\theta_\text{prior})+B_\text{prior}u_k)$ on the data set $\mathcal{D}_\text{known}$\; 
            \vspace{3mm}
            \textbf{Step 2: Adaptive Data-driven MPC}\; 
            $(A,B,C,\theta)\leftarrow (A_\text{prior},B_\text{prior},C_\text{prior},\theta_\text{prior})$  \hspace{3mm}\tcp{\textcolor{blue}{Main network parameters}}
            $(A_\text{target},B_\text{target},C_\text{target},\theta_\text{target})
            \hspace{-1mm}\leftarrow\hspace{-1mm} 
            (A_\text{prior},B_\text{prior},C_\text{prior},\theta_\text{prior})$ 
            \hspace{0mm}\tcp{\textcolor{blue}{Target network params}}
            $(\Theta, \Theta_\text{target})\leftarrow (\{ A,B,C,\theta \}, \{ A_\text{target},B_\text{target},C_\text{target},\theta_\text{target} \})$\;  
            $\mathcal{B}\leftarrow \emptyset$  
            \hspace{3mm}\tcp{\textcolor{blue}{Replay buffer}}
            $x_0\leftarrow $Initial condition\; 
            \vspace{3mm}
	    \For{$k=0,1,2,\cdots$}{
            $u_k\leftarrow \text{LinearMPCSolver}(g(x_k); \Theta_\text{target})$ 
            \hspace{3mm}\tcp{\textcolor{blue}{Koopman MPC}} 
            $x_{k+1}\leftarrow f(x_k,u_k)$  
            \hspace{3mm}\tcp{\textcolor{blue}{Next state from the true environment}}
            $\mathcal{B}\leftarrow \mathcal{B}\cup \{ (x_k,u_k,x_{k+1}) \}$ 
            \hspace{3mm}\tcp{\textcolor{blue}{Add new data to replay buffer}}
            $\mathcal{D}_\text{online}\leftarrow \text{BatchSample}(\mathcal{B})$\; 
            $\Theta \leftarrow \text{GradientDescent}(\mathcal{D}_\text{online})$\; 
            $\Theta_\text{target}\leftarrow \tau \Theta + (1-\tau)\Theta_\text{target}$
            \hspace{3mm}\tcp{\textcolor{blue}{Soft update}}
            }
\end{algorithm2e}
\vspace{-4mm}
\subsection{Parameter Selection for Online Update}
\vspace{-2mm}
\label{sec. parameter selection based on Koopman form}
While the online update can be applied to all the parameters in $\Theta$, it is advisable to only select ones that have dominant and essential effects on the control performance to improve both the computational efficiency and the robustness of the algorithm.
For instance, the number of  parameters in $\Theta$ can become quite large for complex and/or large-scale dynamics, and updating all the parameters may lead to an undesirable computational burden for online updates.
Also, the stability of the online update can be further improved by excluding parameters that are sensitive to model outputs but not meaningful in terms of control performance.
To this end, we propose to exclude the matrix $A$ from online updated parameters if the online learning becomes unstable, which is suggested by the following property of linear embedding models.
\vspace{-1mm}
\begin{thm}
    \label{prop. Koopman form}
    \rm{}
    \hspace{-3mm}(\cite{Koopman_form})\\
    Suppose that $\mathcal{X}$ and $\mathcal{U}$ are convex sets and $0\in\mathcal{U}$. For a linear embedding model \eqref{eq. linear embedding model; model dynamics in the embedded space}, \eqref{eq. linear embedding model; decoder} s.t. $g_i\in C^1$ for $\forall i$ and $\text{span}(g_1,\cdots,g_N)$ is an invariant subspace under the action of the Koopman operator associated with $f(\cdot,0)$, the following holds:
    \vspace{-2mm}
    \begin{align}
		g(x_{k+1})
		=
            A g(x_k)
		+
            \underset{=:\hat{B}(x_k,u_k)}
            {
            \underbrace{
            \int_{0}^{1} \cfrac{\partial \mathcal{B}}{\partial u} (x_k, \lambda u_k)
		d\lambda 
            }}
		\ u_k
	,
	\label{eq. general separate form with Beta relation and A part}
\end{align}
\vspace{-4mm}
    where 
\begin{align}
	&\mathcal{B}(x,u):=
	\left\{
	\int_{0}^{1} \cfrac{\partial g}{\partial x}(f(x,0) + \lambda (f(x,u) - f(x,0)))
	d\lambda\right\}
	(f(x,u) - f(x,0)),&
	\label{eq. Beta}
\end{align}
and $\mathcal{B}(x,u)$ is assumed to be differentiable in $u$.
\end{thm}

Equation \eqref{eq. general separate form with Beta relation and A part} implies that there exists a linear embedding model with no modeling error s.t. $A$ is a constant matrix while $B$ is given as a time-varying one $\hat{B}(x_k,u_k)$. 
Specifically, if we can find a finite-dimensional Koopman operator associated with the drift term of the dynamics, it is sufficient to update only $B$ appropriately at every time step to reconstruct the true dynamics. 
Whereas there is no guarantee in general that the learning results of Problem \ref{prob. Offline Model Learning} or the online update satisfy the assumptions in Theorem \ref{prop. Koopman form}, we heuristically find that excluding $A$ from the online updated parameters while updating $B$ and $g_i$ online improves control performance and stabilizes learning in many cases.
\vspace{-6mm}
\section{Numerical Example}
\vspace{-2mm}
\label{sec. numerical example}
As a numerical example, we consider a cartpole system with a cart mass $m_c$, a pole mass $m_p$, and a pole length $2l$, which is described by the following ordinary differential equations (ODE) (\cite{safe_control_gym}):
\vspace{-2mm}
\begin{align}
         \ddot{x}(t)
    =\cfrac{
        F(t) + m_pl(\dot{\theta}^2(t)\sin \theta(t) - \ddot{\theta}(t)\cos\theta(t) )
    }{m_c + m_p}, 
    \label{eq. cartpole 1}
\end{align}
\begin{align}
          \ddot{\theta}(t)
    =\cfrac{
        g\sin\theta(t) + \cos\theta(t) 
        \left(
            \frac{
                -F-m_p l \dot{\theta}^2(t) \sin\theta(t) 
            }{m_c + m_p}
        \right) 
    }{
        l
        \left(
            \frac{4}{3} - 
            \frac{m_p\cos^2\theta(t)}{m_c+m_p}
        \right) 
    }, 
    \label{eq. cartpole 2}
\end{align}
where $x(t)$, $\theta(t)$, $F(t)$, and $g$ are the cart position, the pole angle, the force to the cart, and the acceleration due to gravity, respectively. 
The state of the system is supposed to be $x_k:=[x(k\Delta t)\ \dot{x}(k\Delta t)\ \theta(k\Delta t)\ \dot{\theta}(k\Delta t)]^\tr$ s.t. $x(t)$, $\dot{x}(t)$, $\theta(t)$, and $\dot{\theta}(t)$ are sampled with a sampling period $\Delta t=1/15$ [s].
Also, $F(t)$ is given by the control input $u_k$ determined by MPC s.t. $F(t):=u_k$ for $k\Delta t\leq t\leq (k+1)\Delta t$.
The reference signal is set to $x^\text{ref}_k:=0.$

It is assumed that we are given the ODE \eqref{eq. cartpole 1}, \eqref{eq. cartpole 2} with $(m_c, m_p, l)=(0.75, 0.075, 0.375)$ as a nominal model, whereas the true dynamics is governed by the same equations with different parameter values: $(m_c, m_p, l)=(1,\ 0.1,\ 0.5)$.
In Step 1 of Algorithm \ref{algorithm proposed method}, we collect data from the nominal model that consists of 500 trajectories, each of which has a length of 60 time steps, and train a linear embedding model with the feature maps given by $g(x_k;\theta):=[x_k^\tr\ g_5(x_k;\theta)\ g_6(x_k;\theta)]^\tr$, where $[g_5(\cdot;\theta)\ g_6(\cdot;\theta)]$ is a feed-forward neural network with three hidden layers, each of which has 64 neurons.
Note that including the state $x_k$ itself in the feature maps eliminates the decoding error and ensures the condition \eqref{eq. condition on linear embedding model; decoding error} by an analytical decoder $C:=[I_n\ 0]$. Therefore, we set $\lambda_2=0$ in \eqref{eq. loss function}.
Following the discussion in Section \ref{sec. parameter selection based on Koopman form}, we fix $A$ to the learning result of Problem \ref{prob. Offline Model Learning} and update $g$ and $B$ online with $\tau:=0.05$.

For comparison, we also test three controllers in addition to the proposed method. 
As a baseline of non-adaptive MPC, we consider the nominal MPC, which is described in Section \ref{sec. problem setup}.
As data-driven adaptive MPC, the Gaussian Process MPC (\textit{GP-MPC}) (\cite{GP_MPC}) and the MPC with random Fourier Features (\textit{RFF-MPC}) (\cite{RFF_MPC}) are considered.
GP-MPC and RFF-MPC learn the residual dynamics $r(x_k,u_k)$ in \eqref{eq. true dynamics decomposition} with sparse Gaussian processes and random Fourier features in online manners, respectively. 
For all the MPC methods, we set $H:=20$, $Q_\text{state}:=\text{diag}(5,0.1,5,0.1)$, and $R:=0.1$.
For each controller, a simulation with the same setting is performed 10 times with randomly chosen initial conditions. The simulations are implemented with safe control gym (\cite{safe_control_gym}) on a system with an AMD Ryzen 7 7700X 8-core processor and 32 GB of memory.

\begin{figure}[b]
	\centering
	\subfigure[Sample trajectories.]{
            \includegraphics[width=0.4\linewidth]{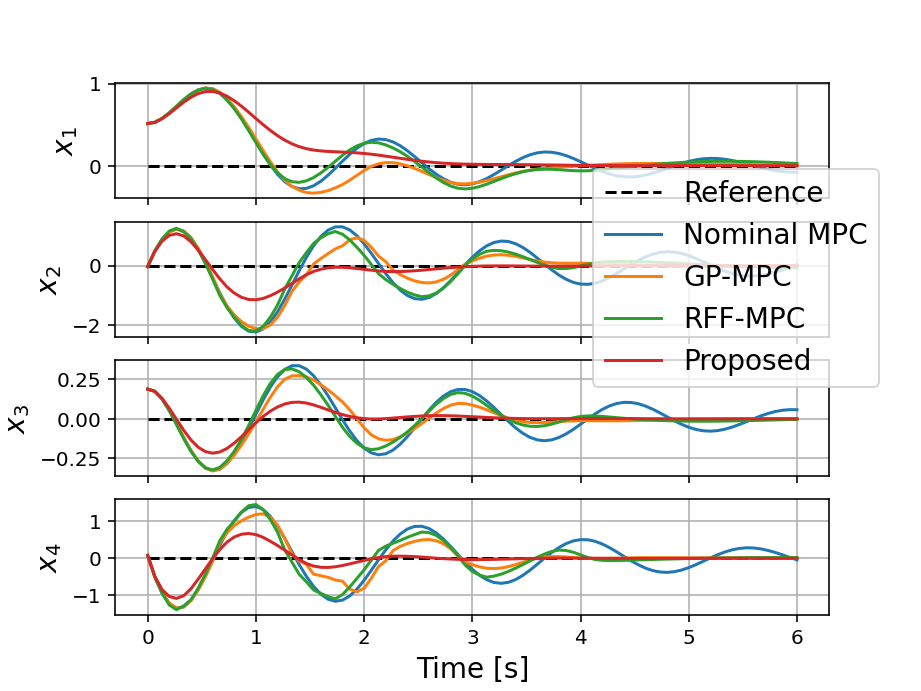}
        }
        \subfigure[Average errors.]{
            \includegraphics[width=0.4\linewidth]{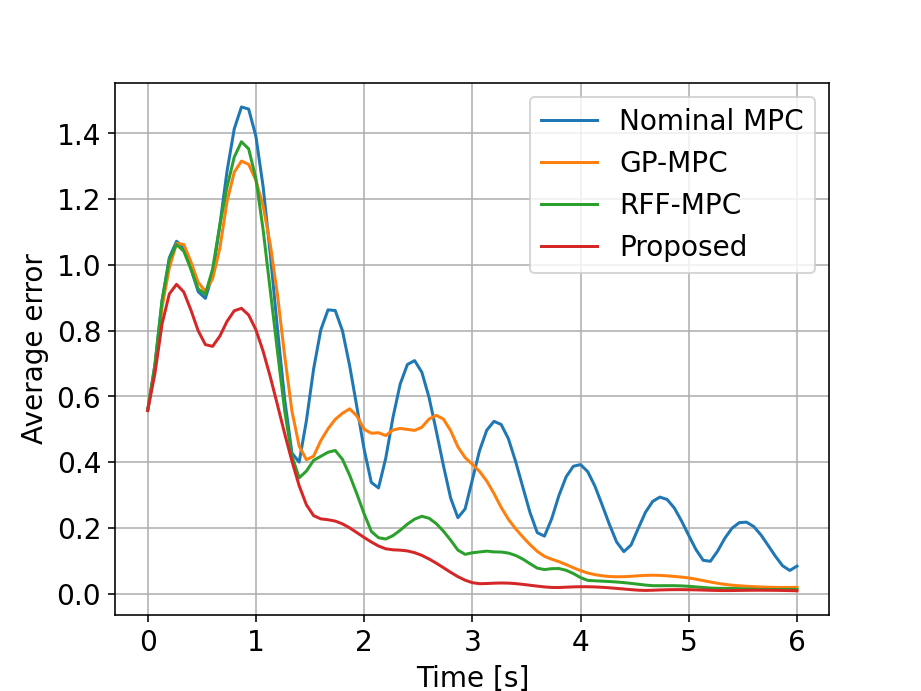}
        }
        \caption{Results of the cartpole system.}
        \label{fig. sample traj and error}
\end{figure}
\begin{figure}[b]
	\centering
	\subfigure[10 \% difference.]{
            \includegraphics[width=0.3\linewidth]{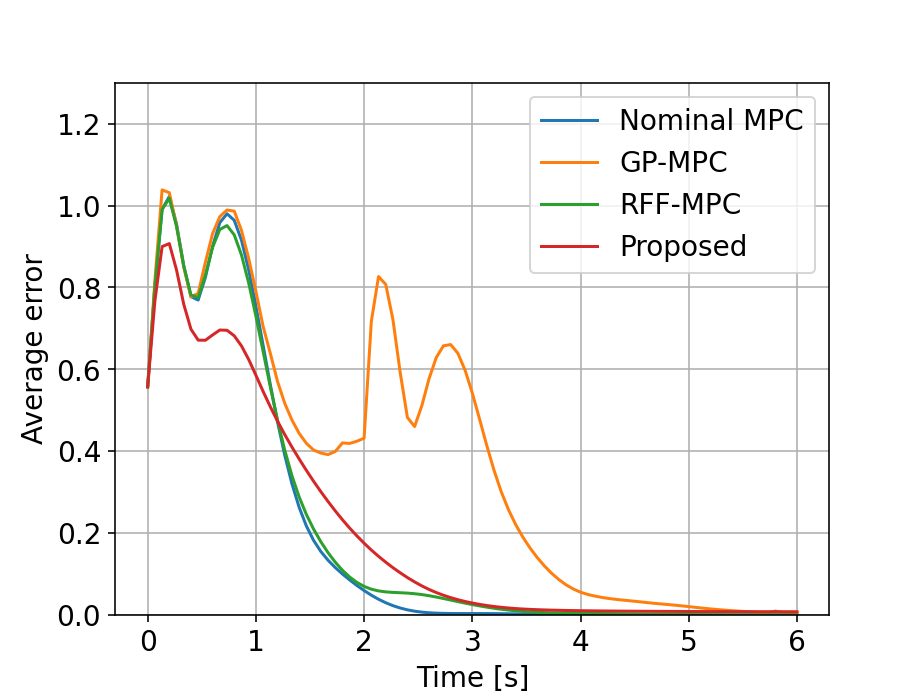}
        }
        \subfigure[20 \% difference.]{
            \includegraphics[width=0.3\linewidth]{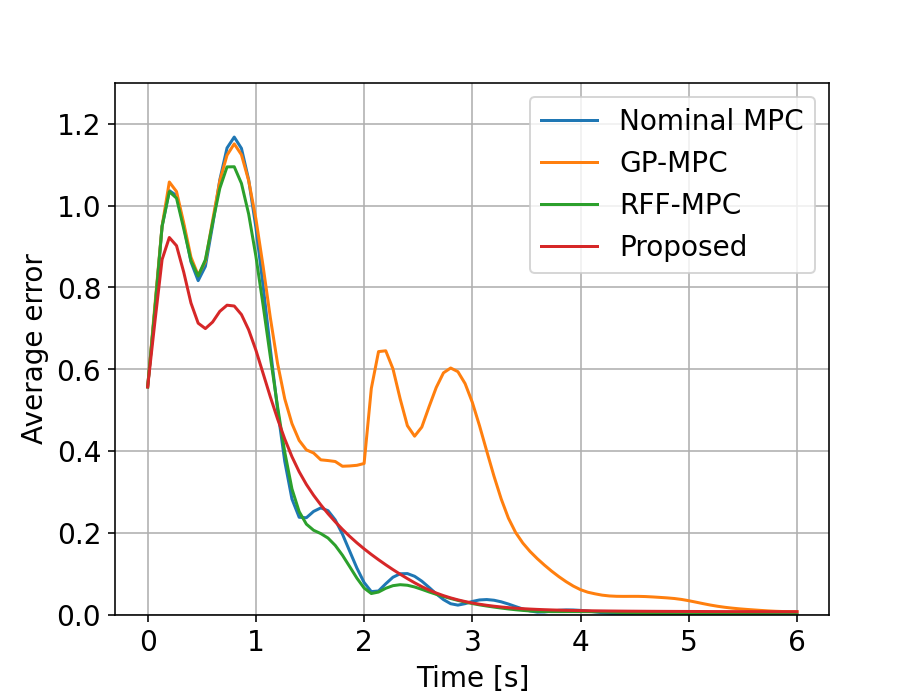}
        }
        \subfigure[30 \% difference.]{
            \includegraphics[width=0.3\linewidth]{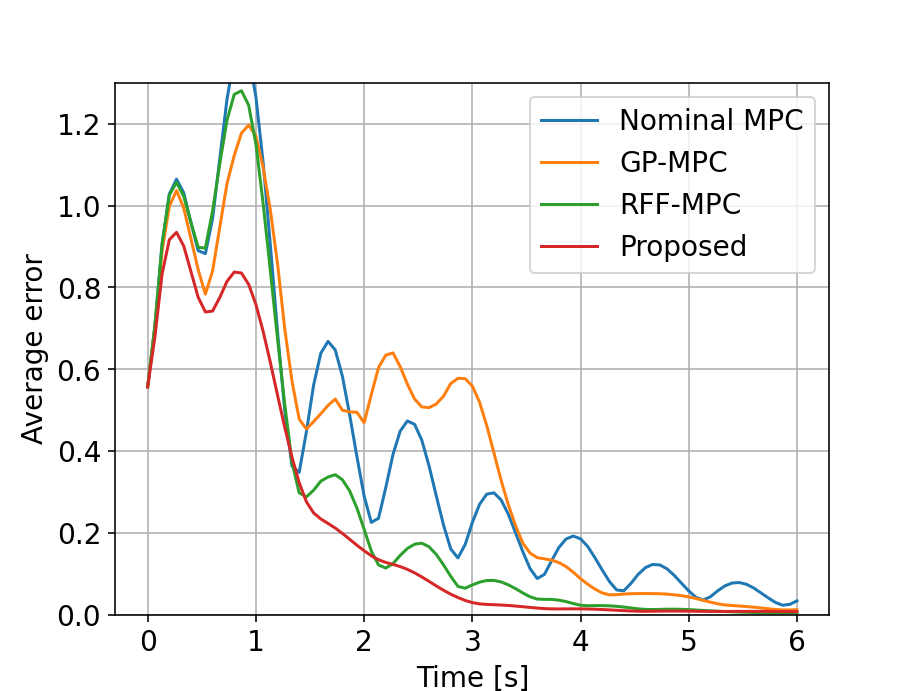}
        }
        \caption{Average errors with various extents of discrepancy between the nominal model and the true dynamics.}
        \label{fig. sensitivity study}
\end{figure}

The results are shown in Fig. \ref{fig. sample traj and error}, where a sample trajectory and the average error defined by $\frac{1}{10}\sum_{i=1}^{10} \| x_{k,i}-x^\text{ref}_k \|_2$ ($x_{k,i}$: state of the $i$-th trajectory) for each method are on the left and right panels, respectively. 
Since nominal MPC does not take the effect of the residual dynamics into account, it does not  track the reference within the simulation window of six seconds. 
On the other hand, all the adaptive MPC methods successfully stabilize the states by the end of the simulations.
Compared to GP- and RFF-MPC, the proposed adaptive Koopman MPC outperforms in terms of the average errors.
Table \ref{table. execution times} shows the average execution times of the simulation for individual controllers. 
Since GP- and RFF-MPC are adaptive methods based on nonlinear MPC, they take longer than the nominal MPC.
On the other hand, the proposed adaptive Koopman MPC achieves an even shorter execution time than the nominal MPC thanks to the convexity of its formulation.

Finally, a sensitivity analysis w.r.t. the residual dynamics is performed, where we consider various extents of the discrepancy between the nominal model and the true dynamics.
Figure \ref{fig. sensitivity study} shows the results, where the true dynamics parameters are varied by 10 to 30 \% w.r.t. the nominal model and the same experiment is performed for each case. 
Whereas the proposed method results in higher average errors after $t=1.5$ [s] for relatively small extents of the residual dynamics (Figs. \ref{fig. sensitivity study}(a,b)), it outperforms the other controllers at the beginning of the simulation in all cases.
Also, the proposed method shows the most robust performance across the given range of residual dynamics.

\begin{table}[]
    \centering
    \begin{tabular}{c|cccc}
         Method & Nominal MPC & GP-MPC & RFF-MPC & Proposed 
         \\ \hline 
         Average execution time [s] & 0.70 & 5.11 & 1.12 & 0.52
    \end{tabular}
    \caption{Average execution times of the MPC simulation.}
    \label{table. execution times}
\end{table}

\vspace{-4mm}
\section{Conclusion}
\vspace{-3mm}

This work introduced an adaptive Model Predictive Control (MPC) framework that leverages Koopman operators for nonlinear dynamics. While MPC is widely adopted across diverse control applications, its performance can deteriorate when the control model fails to accurately capture the true system dynamics. To address this challenge, we developed an approach that combines the Koopman operator framework with a linear embedding model, enabling online parameter updates to compensate for model discrepancies. The offline training learns the features and operator matrices jointly, allowing for greater expressivity. By maintaining linearity in the control model while accommodating nonlinear dynamics, our method achieves both computational efficiency and adaptability.
Online learning may become unstable if the model is parameterized by complex model forms such as neural networks, which result in a high-dimensional non-convex optimization. 
The proposed method uses the soft update of target networks so that abrupt changes in the model will be avoided and we can stabilize online updates. 
Also, we provide a discussion on which model parameters to prioritize for the online update based on a specific system description of linear embedding models in \cite{Koopman_form}. 
Experimental validation on a cartpole system demonstrates that our method achieves favorable control performance while requiring significantly lower computational resources compared to existing adaptive MPC approaches. Furthermore, the results reveal enhanced robustness to residual dynamics. While our initial results demonstrate  promise and basic viability, expanding the validation to more complex industrial and robotic systems represents an important next step in establishing this framework's broad applicability. 

\acks{
This work has benefited from several discussions and code-sharing with Hongyu Zhou and Vasileios Tzoumas, for which we are grateful. This work was funded by AFOSR grant FA9550-17-1-0195.
}

\bibliography{l4dc_2025}  
\end{document}